\documentclass[12pt]{amsart}
\usepackage{amsthm}
\usepackage{amsmath}
\usepackage[all]{xy}
\usepackage{amsfonts}
\usepackage{amssymb}
\usepackage{amscd}
\newtheorem{theorem}{Theorem}

\newtheorem{lemma}{Lemma}
\begin{document}

\title[Sums of Powers of Primes]{Sums of Powers of Primes II}

\author{Lawrence C. Washington}
\address{Dept. of Mathematics, Univ. of Maryland, College Park, MD, 20742 USA}\email{lcw@umd.edu}
\maketitle

\begin{abstract} For a real number $k$, define $\pi_k(x) = \sum_{p\le x} p^k$. When $k>0$, we prove that
$$
\pi_k(x) - \pi(x^{k+1}) = \Omega_{\pm}\left(\frac{x^{\frac12+k}}{\log x} \log\log\log x\right)
$$
as $x\to\infty$, and we prove a similar result when $-1<k<0$. This strengthens a result in a paper by J. Gerard and the author and it corrects a flaw in a proof in that paper.
We also quantify the observation from that paper that $\pi_k(x) - \pi(x^{k+1})$ is usually negative when $k>0$ and usually positive when $-1<k<0$.
\end{abstract}

For a real number $k$, define 
$$
\pi_k(x) = \sum_{p\le x} p^k,  \qquad  \psi_k(x)= \sum_{n\le x} \Lambda(n) n^k,
$$
where the first sum runs over prime numbers and $\Lambda(n) = \log p$ if $n=p^m$ (with $m\ge 1$) is a prime power, and $\Lambda(n) = 0$ otherwise. Then $\pi_0(x) = \pi(x)$, the number of primes less than or equal to $x$, and $\psi_0(x)=\psi(x)$ is the Chebyshev function ({\it Note:} There are other uses of the notation $\psi_k$ in the literature that are different from the present one). As usual, we use the notation $f(x) = \Omega_{\pm}(g(x))$ to indicate that there exists a constant $c>0$ such that $f(x)> c g(x)$ for a sequence of $x\to\infty$ and $f(x) < -cg(x)$ for a sequence of $x\to\infty$. Our goal is to prove the following:
\begin{theorem} Let  $\theta_0$  be the supremum of the real parts of the zeros of $\zeta(s)$, and let $\epsilon>0$. As $x\to\infty$,
$$
\pi_k(x) - \pi(x^{k+1}) = \begin{cases}
\Omega_{\pm} (x^{\theta_0+k-\epsilon}) \quad \text{ if } k>0,\\
\Omega_{\pm}( x^{(k+1)(\theta_0-\epsilon)}) \quad \text{ if } -1< k < 0 \text{ and } \theta_0<1.
\end{cases}
$$

If the Riemann Hypothesis is true,
$$
\pi_k(x) - \pi(x^{k+1}) = \begin{cases} \Omega_{\pm}\left(\frac{x^{k+\frac12} \log\log\log x}{\log x}\right) \quad \text{ if } k>0,\\
 \Omega_{\pm}\left(\frac{x^{\frac{k+1}2}\log\log\log x}{\log x} \right) \quad \text{ if } k<0.\end{cases}
$$
\end{theorem}
We are not able to treat the case where $\theta_0=1$ and $k<0$. 

A paper of J.  Gerard and the author \cite{GW} showed that $\pi_k(x)$ is asymptotic to $\pi(x^{k+1})$
(see also \cite{Ja} and \cite{SZ}) 
and proved the first half of the theorem. Unfortunately, this proof was based on an incorrect formula.
When the correct formula is used, the proof in \cite{GW} is valid only when the Riemann Hypothesis is false. The main work of the present paper uses a technique of Littlewood to establish the second half of the theorem, namely under the assumption that the Riemann Hypothesis is true.
The second half implies the first half when $\theta_0=1/2$. 

In \cite{GW}, a heuristic explanation was given for why $\pi_k(x) - \pi(x^{k+1})$ is usually negative when $k>0$ and usually positive when $-1<k<0$.
In Sections 2 and 3 of the present paper, the following more quantitative results are proved. The methods in these sections
were inspired by \cite{Jo}.
\begin{theorem}
$$
\int_1^{\infty} \frac{\pi_k(t) - \pi(t^{k+1})}{t^{k+2}} dt = \frac{-1}{k+1}\log(k+1) \quad \begin{cases} <0 \text{ if } k>0, \\  >0 \text{ if } -1 < k < 0.\end{cases}
$$
\end{theorem}
 
 \begin{theorem}
Let $0<k\le 10.32$. The following are equivalent:
\begin{enumerate}
\item The Riemann Hypothesis.
\item
$$
\int_1^x \left(\pi_k(t) - \pi(t^{k+1})\right) dt  < 0 \qquad \text{for all $x$ sufficiently large.}
$$

\end{enumerate}
\end{theorem}

\section{Proof of Theorem 1}
\begin{proof}
When $k<0$ and $\theta_0<1$, the result was proved in \cite[p. 174]{GW}.
It was also proved that
$$
\pi_k(x)- \pi(x^{k+1}) = -E(x^{k+1})+x^kE(x) - k\int_2^x t^{k-1}E(t) dt  + O(1),
$$
where $E(x) = \pi(x) - \text{li}(x)$.  When the Riemann Hypothesis holds, the estimate $E(y)= O(y^{1/2}\log y)$  (see \cite[Theorem 13.1]{MV}), combined with Littlewood's oscillation result (see below, or \cite[p. 479]{MV}),
yields the stronger statement given in the second half of the theorem:
\begin{align*}
\pi_k(x)- \pi(x^{k+1}) &=\Omega_{\pm}\left(\frac{x^{\frac{k+1}{2}}\log\log\log x}{\log x}\right) +x^kO\left(x^{1/2}\log x\right) \\
&+O\left(\int_2^x t^{k-1}t^{1/2}\log t dt\right)  + O(1)\\
&= \Omega_{\pm}\left(\frac{x^{\frac{k+1}{2}}\log\log\log x}{\log x}\right),
\end{align*}
since $(k+1)/2 > k+\frac12$ when $k<0$.

In the above, we have used the following well-known lemma. Since we also use it several times in the following, we state it explicitly.

\begin{lemma} Let $r$ be a real number (positive or negative) and let $\ell>-1$. Then
$$
\int_2^x t^{\ell}\log^r t\,  dt = \frac{1}{\ell+1}x^{\ell+1}\log^r x + O\left(x^{\ell+1}\log^{r-1} x\right)
$$
as $x\to\infty$.
\end{lemma}
\begin{proof}
$$
\int_2^x t^{\ell}\log^r t \, dt = \frac{1}{\ell+1}x^{\ell+1}\log^r x + O(1)  - \frac{1}{\ell+1}\int_2^x r t^{\ell}\log^{r-1}t\,  dt.
$$
 When $\ell > 0$ and $t$ is sufficiently large, $t^{\ell}\log^{r-1} t$ is increasing, so 
$$
\int_2^x  t^{\ell}\log^{r-1} t\, dt \le x^{\ell+1}\log^{r-1} x,
$$
for large $x$.

Now suppose $-1<\ell\le 0$. Choose $0<\alpha<1$ and choose $\epsilon$ with $0<\epsilon < (\ell+1)(1-\alpha)/\alpha$. Then $\alpha(\ell+1+\epsilon) <\ell+1$, so
when $x$ is sufficiently large, 
\begin{align*}
\int_2^x  t^{\ell}\log^{r-1} t\, dt &= \int_2^{x^{\alpha}} t^{\ell}\log^{r-1} t\, dt+\int_{x^{\alpha}}^x t^{\ell}\log^{r-1}t\, dt\\
&= O\left( \int_2^{x^{\alpha}} t^{\ell+\epsilon} dt\right) +O\left((\log^{r-1} x)\int_{x^{\alpha}}^x t^{\ell} dt\right)\\
& = O\left(x^{\alpha(\ell+1+\epsilon)}\right) +  O\left(x^{\ell+1}\log^{r-1} x\right)\\
&=   O\left(x^{\ell+1}\log^{r-1} x\right).
\end{align*}
\end{proof}

For the remainder of this section, we are interested in $k>0$, although we state the lemmas in forms that hold for $k>-1$.

Let
\begin{align}
 \Pi_k(x) &= \pi_k(x) + \frac12\pi_{2k}(x^{1/2})  + \frac13 \pi_{3k}(x^{1/3}) + \frac14\pi_{4k}(x^{1/4}) + \cdots \label{Correct}\\
 &=    \sum_{n\le x} \frac{\Lambda(n) n^k}{\log n}.
\end{align}
In \cite{GW}, the formula (\ref{Correct}) was given incorrectly (it used $\pi_k(x^{1/2}), \pi_k(x^{1/3})$, etc. instead of $\pi_{2k}(x^{1/2}), \pi_{3k}(x^{1/3})$, etc.).

\begin{lemma}\label{Lemma1}
Let $k>-1$. Then
$$
\Pi_k(x) = \pi_k(x) +\frac12 \pi_{2k}(x^{1/2}) + O(x^{k+\frac13}) + O(\log x)
$$
and
$$
\Pi_k(x) = \pi_k(x) + O\left(\frac{x^{k+\frac12}}{\log x}\right) + O(\log x).
$$
\end{lemma}
\begin{proof}
There are $O(\log x)$ positive integers $\ell$ such that $x^{1/\ell}\ge 2$, so the sum in (\ref{Correct}) has $O(\log x)$ nonzero terms. 
When $k\ge 0$, we have
$$
\pi_{k}(y) \le \pi(y) y^k \le 1.3\frac{y^{k+1}}{\log y}
$$
(the first inequality is the trivial estimate and the second is in \cite{RS}).
Therefore,  
$$\frac13\pi_{3k}(x^{1/3})+\frac14\pi_{4k}(x^{1/4}) + \cdots = O\left(\frac{x^{k+\frac13}}{\log x}\log x\right) = O(x^{k+\frac13}).$$

Now suppose $-1<k<0$. Let $n$ be the largest positive integer such that $nk\ge -1$. Since $(n+1)k<-1$, we have $\pi_{mk}(y)\le \pi_{(n+1)k}(y) = O(1)$ for all $m\ge n+1$, and at most $O(\log x)$ terms in (\ref{Correct}) are nonzero.
Moreover, $\pi_{2k}(y)\sim \pi(y^{2k+1})$, from which it follows that
\begin{align*}
\Pi_k(x) &= \pi_k(x) + \frac12\pi_{2k}(x^{1/2}) \\
&+ \frac{x^{k+\frac13}}{(3k+1) \log(x)} + o\left(\frac{x^{k+\frac13}}{(3k+1) \log(x)}\right) \\
&+ \cdots + \frac{x^{k+\frac1n}}{(nk+1) \log(x)}+ o\left(\frac{x^{k+\frac1n}}{(nk+1) \log(x)}\right)  +  O(\log x)\\
&= \pi_k(x) + \frac12\pi_{2k}(x^{1/2})+ O\left(\frac{x^{k+\frac13}}{\log x}\right) + O(\log x),
\end{align*}
as desired (the implied constants depend on $k$, but not on $x$).

The second equality of the lemma follows from the first equality and the fact that $\pi_{2k}(y)\sim \pi(y^{2k+1})$.
\end{proof}

In \cite{GW}, it was proved that $\Pi_k(x)-\Pi_0(x^{k+1}) = \Omega_{\pm} (x^{\theta_0 + k -\epsilon})$. When $k>0$, Lemma \ref{Lemma1} allows us to change $\Pi_k(x)$ and $\Pi_0(x^{k+1})$ to $\pi_k(x)$ and $\pi(x^{k+1})$ with errors
of $o(x^{k+\frac12})$ and $o(x^{\frac{k+1}2})$, respectively. These are dominated by the oscillation term if the Riemann Hypothesis is false.

Henceforth, we assume the Riemann Hypothesis (RH)  is true and deduce Theorem 1 in this case.

\begin{lemma}\label{Lemma2} Let $k$ be a real number and let $c> k+1$. If $x>0$, then
$$
\int_2^x \psi_k(t)\, dt = \frac{-1}{2\pi i} \int_{c-i\infty}^{c+i\infty} \frac{\zeta'(s-k)}{\zeta(s-k)} \frac{x^{s+1} ds}{s(s+1)}.
$$
\end{lemma}
\begin{proof} The proof is identical to the proof in \cite[pp. 31-32]{I}.
\end{proof}
\begin{lemma}\label{Lemma3} Assume RH. Let 
$$
A=\text{Res}_{s=0}\left( \frac{\frac{\zeta'(s-k)}{\zeta(s-k)}}{s(s+1)} \right), \qquad B=\text{Res}_{s=-1}\left( \frac{\frac{\zeta'(s-k)}{\zeta(s-k)}}{s(s+1)} \right).
$$
If $k>-1$, then 
$$
\int_2^{x} \psi_k(t)\, dt = \frac{x^{k+2}}{(k+1)(k+2)} - \sum_{\rho} \frac{x^{\rho+k+1}}{(\rho+k)(\rho+k+1)} -Ax+B+O(x^{k-1})
$$
as $x\to\infty$. The sum is over the zeros $\rho$ of $\zeta(s)$ with $\text{Re}(\rho)=1/2$, counted with multiplicity.
\end{lemma}
\begin{proof}
The proof proceeds by moving the line of integration to the left. The details are the same as in \cite[pp. 73-74]{I}, where the proof is given when $k=0$.
\end{proof}

\begin{lemma}\label{Lemma4} Assume RH and let $k>-1$. Then
$$
\int_2^{x} \psi_k(t)dt = \sum_{n\le x} (x-n)\Lambda(n) n^k = \frac{x^{k+2}}{(k+1)(k+2)} + O(x^{k+\frac32})+ O(x).
$$
\end{lemma}
\begin{proof}
The first equality is proved by integration by parts:
$$
\int_2^x \psi_k(t) dt= x\psi_k(x) - \int_{2-}^x t\, d\psi_k(t)  = x\sum_{n\le x} \Lambda(n) n^k - \sum_{n\le x} n^{k+1} \Lambda(n),
$$
which yields the result. 

The second equality follows from Lemma \ref{Lemma3} and the absolute convergence of $\sum 1/\rho^2$.
\end{proof}

\begin{lemma}\label{Lemma5} Assume RH and let $k>-1$. Then 
$$
\pi_k(x)- \pi(x^{k+1}) = \frac{\psi_k(x)-\frac{x^{k+1}}{k+1}}{\log x}  + O(x^{\frac{k+1}2} \log x) + O\left(\frac{x^{k+\frac12}}{\log x}\right).
$$
\end{lemma}

\begin{proof}
We first prove the lemma with $\Pi_{k}(x)$ in place of $\pi_k(x)$:
\begin{align*}
\Pi_k(x) &= \int_{2-}^x\frac {d\psi_k(t)}{\log t} \\
&= \frac{\psi_k(x)}{\log x} + \int_2^x \frac{\psi_k(t)-\frac{t^{k+1}}{k+1}}{t\log^2t} dt + \frac{1}{k+1}\int_2^x \frac{t^k}{\log^2 t} dt\\
&= \frac{\psi_k(x)}{\log x} + I_1 + I_2.
\end{align*}
Integration by parts yields
\begin{align*}
I_1 &= \frac{\sum_{n\le x}  (x-n)\Lambda(n) n^k - \frac{x^{k+2}}{(k+1)(k+2)}}{x \log^2 x} + O(1)\\
& + \int_2^x \frac{\sum_{n\le t} (t-n) \Lambda(n) n^k - \frac{t^{k+2}}{(k+1)(k+2)}}{t^2\log^2 t} \left(1+\frac{2}{\log t}\right) dt\\
&= O\left(\frac{x^{k+\frac12}}{\log^2 x} \right) + O(1)+ \int_2^x \frac{O(t^{k+\frac32})+O(t)}{t^2 \log^2 t} dt \quad (\text{by Lemma \ref{Lemma4}})\\
&= O\left(\frac{x^{k+\frac12}}{\log^2 x}\right)+ O(1).
\end{align*}
Integration by parts also yields
\begin{align*}
I_2 &= \frac{1}{k+1} \int_2^x \frac{t^{k+1}}{t \log^2 t} dt \\
&= \frac{-x^{k+1}}{(k+1)\log x} + \int_2^x \frac{t^k}{\log t} dt + O(1)\\
&= \frac{-x^{k+1}}{(k+1)\log x} + \int_2^ {x^{k+1}} \frac{du}{\log u} + O(1) \quad \text{ (substitute $u=t^{k+1}$)}\\
&= \frac{-x^{k+1}}{(k+1) \log x} + \pi(x^{k+1}) + O\left(x^{\frac{k+1}{2}}\log x\right). 
\end{align*}
In the last equality, we have used the fact \cite[Theorem 13.1]{MV} that the Riemann Hypothesis implies $\pi(y)= Li(y)+ O(y^{1/2}\log y)$. Putting everything together yields 
\begin{equation}\label{PiandPsi}
\Pi_k(x) =  \frac{\psi_k(x)-\frac{x^{k+1}}{k+1}}{\log x} + \pi(x^{k+1}) + O(x^{\frac{k+1}2} \log x)+ O\left(\frac{x^{k+\frac12}}{\log^2 x}\right).
\end{equation}

Since Lemma \ref{Lemma1} tells us that
$$
\Pi_k(x) =  \pi_k(x) + O\left(\frac{x^{k+\frac12}}{\log x}\right) + O(\log x),
$$
the result of the lemma follows.
\end{proof}

Lemma \ref{Lemma5} translates oscillations of $\psi_k(x)$ into oscillations of $\pi_k(x) - \pi(x^{k+1})$. We now use a method of Littlewood, as modified by Ingham, to produce oscillations in $\psi_k(x)$.
The following lemma allows us to look at $\psi_k(x)$ averaged over a small interval. 

\begin{lemma}\label{Lemma6} Assume RH and let $k>-1$.  Then, uniformly for $x\ge 4$ and $\frac{1}{2x} \le \delta\le \frac12$, 
\begin{align*}
&\frac{1}{(e^{\delta} - e^{-\delta})x} \int_{e^{-\delta}x}^{e^{\delta}x} \left(\psi_k(u) - \frac{u^{k+1}}{k+1}\right) du \\  \\
&= -2x^{k+\frac12} \sum_{\gamma > 0} \frac{\sin(\gamma\delta)}{\gamma\delta}
\frac{\sin(\gamma \log x)}{\gamma} + O(x^{\frac12 + k}).
\end{align*}
\end{lemma}
\begin{proof} The proof follows from Lemma \ref{Lemma3}, as in the proof of  \cite[Lemma 15.9]{MV}. \end{proof}

Finally, the proof of Theorem 1 can now be completed using a Diophantine approximation argument, as in the proof of 
\cite[Theorem 15.11]{MV}. In particular,  suitable large values of $x$ coupled with small values  $\delta$ can be found to show that the sum of the right side of the formula in Lemma \ref{Lemma6} is $\Omega_{\pm}(\log\log\log x)$.
Since the left side is the average over an interval, we find that 
$$
\psi_k(x) - \frac{x^{k+1}}{k+1} = \Omega_{\pm} \left( x^{k+\frac12}\log\log\log x\right).
$$
When $k>0$ in Lemma \ref{Lemma5}, the oscillation term dominates. This  yields the desired result for $\pi_k(x) - \pi(x^{k+1})$ and completes the proof of Theorem 1. \end{proof}

\section{Proof of Theorem 2}

\begin{proof}
Recall that
$$
\sum_{p\le x} \frac1p = \log\log x + B + o(1),
$$
where $B= 0.261497\dots$ is a constant \cite{RS}.

The substitution $u=t^{k+1}$ yields
\begin{align*}
\int_1^x \frac{\pi(t^{k+1})}{t^{k+2}} dt &= \frac{1}{k+1} \int_1^{x^{k+1}}\frac{\pi(u) }{u^2} du\\
&= \frac{-1}{k+1}\frac{\pi(x^{k+1})}{x^{k+1}} + \frac{1}{k+1} \int_1^{x^{k+1}} \frac{1}{u} d\pi(u)\\
&=\frac{-1}{k+1}\frac{\pi(x^{k+1})}{x^{k+1}} +\frac{1}{k+1}\sum_{p\le x^{k+1}} \frac1p \\
&= \frac{1}{k+1} \log\log (x^{k+1}) + \frac{B}{k+1} + o(1).
\end{align*}
Similarly,
\begin{align*}
\int_1^x \frac{\pi_k(t)}{t^{k+2}} dt &=  \frac{-1}{k+1} \frac{\pi_k(x)}{x^{k+1}} + \frac{1}{k+1}\int_1^x \frac{1}{t^{k+1}}d\pi_k(t)\\
&= \frac{-1}{k+1} \frac{\pi_k(x)}{x^{k+1}} + \frac{1}{k+1}\sum_{p\le x} \frac{1}{p^{k+1}}p^k\\
&= \frac{1}{k+1} \log\log (x) + \frac{B}{k+1} + o(1).
\end{align*}
Therefore,
$$
\int_1^x \frac{\pi_k(t) - \pi(t^{k+1})}{t^{k+2}} dt = \frac{-1}{k+1} \log(k+1) + o(1).
$$
This yields the theorem.\end{proof}

\section{Proof of Theorem 3}

\begin{proof} Assume the Riemann Hypothesis is true. 
We need the following technical result:
\begin{lemma}\label{Lemma7} Assume RH and let $k>0$. Then
$$
\left|\int_2^x\frac{\psi_k(t) - \frac{t^{k+1}}{(k+1)}}{\log t}dt\right| < 0.04621\frac{x^{k+\frac32}}{\log x}
$$
for all sufficiently large $x$.
\end{lemma}
\begin{proof}
From \cite[Corollary 1]{BPT}, with the standard notation $\rho=\frac12+i\gamma$ for zeros of $\zeta(s)$ in the critical strip, we know that
$$
\sum_{\rho} \left| \frac{1}{(\rho+k)(\rho+k+1)} \right| < \sum_{\rho} \frac{1}{\gamma^2} < 0.04620999.
$$
Let $D(x) = \psi_k(x) - \frac{x^{k+1}}{(k+1)}$. Integration by parts yields
\begin{align*}
\int_2^x \frac{D(t)}{\log t} dt &= \frac{\int_2^x D(t)dt}{\log x} + \int_2^x \frac{\int_2^t D(u) du}{t\log^2 t} dt\\
&< 0.04620999 \frac{x^{k+\frac32}}{\log x}  + 0.04620999 \int_2^x \frac{t^{k+\frac12}}{\log^2 t} dt + o(x) +o(x^{k-1})\\
&(\text{from Lemma \ref{Lemma3}})\\
& < 0.04621  \frac{x^{k+\frac32}}{\log x} 
\end{align*}
when $x$ is large.
\end{proof}

We know that
$$
\Pi_k(x)=\pi_k(x)+\frac{x^{k+\frac12}}{(2k+1)\log x} + o\left(\frac{x^{k+\frac12}}{\log x}\right),
$$
which implies that
\begin{align*}
\int_1^x \left(\Pi_k(t)-\pi_k(t)\right) dt &= \int_2^x \frac{t^{k+\frac12}}{(2k+1)\log t} + o\left(\int_2^x \frac{t^{k+\frac12}}{\log t}\, dt\right) \\
&= (1+o(1)) \frac{x^{k+\frac32}}{(k+\frac32)(2k+1)\log x} 
\end{align*}
as $x\to\infty$.
From Lemma \ref{Lemma7} and (\ref{PiandPsi}),
$$
\int_1^x \left(\Pi_k(t) - \pi(t^{k+1})\right) dt \le  (0.04621+o(1)) \frac{x^{k+\frac32}}{(k+\frac32)\log x}
$$
for $x$ sufficiently large.
Therefore, 
$$
\int_1^x \left(\pi_k(t) - \pi(t^{k+1})\right) dt \le \left(0.04621-\frac{1}{2k+1}+o(1)\right) \frac{x^{k+\frac32}}{(k+\frac32)\log x} < 0
$$
for $x$ sufficiently large, when $k\le 10.32$.
\medskip

\noindent
{\bf Remark.} It would be possible to improve the upper bound 10.32 slightly by using knowledge of the first several values of $\gamma$ to estimate
the early terms of $\sum 1/(\rho+k)(\rho+k+1)$, but it follows from a theorem of Lehman (see \cite{BPT}) that this sum is approximately a constant times $\log(k)/k$, hence is larger than $1/(2k+1)$
for large $k$.
\medskip

Now suppose the Riemann Hypothesis is false. Then  
$$
\theta_0 = \sup\{\text{Re}(\rho) \, | \, \zeta(\rho)=0\} > \frac12.$$
 Let $0<\epsilon <\theta_0-\frac12$ and let 
$$
F(x) = \int_1^x \left(\Pi_k(t) - \Pi_0(t^{k+1})\right) dt.
$$
Then $F(x) = o\left(x^{k+2}\right)$ as $x\to\infty$.
Let $s\in \mathbb C$ with $\text{Re}(s) > k+1$. Then
\begin{align*}
G(s):=&\int_1^{\infty} \left(\frac{F(t) - t^{\theta_0+k+1-\epsilon}}{t^{s+2}}\right) dt\\
 &=\left.\frac{-F(t)}{(s+1)t^{s+1}}\right|_1^{\infty} + \frac{1}{s+1} \int_1^{\infty} \frac{\Pi_k(t) - \Pi_0(t^{k+1})}{t^{s+1}} dt\\
&  +\frac{1}{\theta_0+k-s-\epsilon}    \\
&= \frac{1}{s(s+1)}\left( \log\zeta(s-k) - \log\zeta(\frac{s}{k+1})\right)  +\frac{1}{\theta_0+k-s-\epsilon} .
\end{align*}
The last equality follows from a second integration by parts plus a change of variables in the second part of the integral; see \cite[p.  173]{GW}.
The last line represents a function that is analytic for all real numbers $s> \theta_0+k-\epsilon$. 

Let $\rho$ be the zero of $\zeta(s)$ with $\text{Re}(\rho)>\theta_0-\epsilon$ with smallest positive imaginary part. 
Then $\log\zeta(s-k)$ is not analytic at $\rho+k$. Since $\theta_0\le 1$,
$$
\theta_0 - \epsilon < \text{Re}(\rho) \le \text{Re}\left(\frac{\rho+k}{k+1}\right).
$$
Since the imaginary part of $(\rho+k)/(k+1)$ is less than the imaginary part of $\rho$, the choice of $\rho$ implies that $\zeta((\rho+k)/(k+1))\ne 0$ and therefore does not cancel the singularity at $\rho+k$.
Therefore, $G(s)$ is not analytic at $\rho+k$.

If $F(x) - x^{\theta_0+k+1-\epsilon}\le 0$ for all sufficiently large $x$, Landau's Theorem (see, for example, \cite{GW} or \cite{I}) implies that  $G(s)$ is analytic for all $s\in \mathbb C$ with
$\text{Re}(s) > \theta_0+k-\epsilon$. Since
$$
\text{Re}(\rho+k)> \theta_0+k-\epsilon ,
$$
this is a contradiction. Therefore, there is a sequence of $x\to \infty$ with
$$
 \int_1^x \left(\Pi_k(t) - \Pi_0(t^{k+1})\right) dt > x^{\theta_0+k+1-\epsilon}.
$$
We now need to change from $\Pi_k$ to $\pi_k$.  Since $\Pi_k(t) = \pi_k(t) + O(t^{k+\frac12}/\log x)$ and $\Pi_0(t^{k+1}) = \pi(t^{k+1}) + O(t^{\frac{k+1}{2}}/\log t)$, 
$$
\left| \int_1^x \left(\Pi_k(t) - \Pi_0(t^{k+1})\right)  dt -  \int_1^x \left(\pi_k(t) - \pi_0(t^{k+1})\right)  dt \right| = O( x^{k+\frac32}/\log x).
$$
Since $\theta_0+k+1-\epsilon > k+\frac32$, we find that there exists a sequence of $x\to\infty$ such that
$$
 \int_1^x \left(\pi_k(t) - \pi_0(t^{k+1})\right) dt >  0.
$$
\end{proof}
We note that the proof of ``(2) $\implies$ (1)'' in Theorem 3 works for all $k>0$.

\medskip

\noindent
{\bf Acknowledgement.}

Many thanks to Rusen Li for pointing out the mistake in \cite{GW}.


\begin{thebibliography}{10}
\bibitem{BPT} R. Brent, D. Platt, and T. Trudgian, Accurate estimation of sums over zeros of the Riemann zeta-function,, {\it Math. Comp.} 90 (2021), no. 332, 2923--2935.
\bibitem{GW} J. Gerard and L. C. Washington, Sums of powers of primes, {\it Ramanujan J.} 45 (2018), no. 1, 171--180.
\bibitem{I} A. E. Ingham, {\it The Distribution of Prime Numbers}, Cambridge University Press, 1990.
\bibitem{Ja} R. Jakimczuk,  Desigualdades y formulas asint\'oticas para sumas de potencias de primos, {\it Bol. Soc. Mat. Mexicana} (3) 11 (2005), no. 1, 5–10.
\bibitem{Jo} D. Johnston, On the average value of $\pi(t)-li(t)$, {\it Canadian Math. Bulletin}, 2022.
\bibitem{MV} H. L. Montgomery and R. C. Vaughan, {\it Multiplicative Number Theory I: Classical Theory}, Cambridge University Press, 2007.
\bibitem{RS} B. Rosser and L. Schoenfeld,  Approximate Formulas for Some Functions of Prime Numbers, {\it Illinois J. Math.} 6 (1962), 64--94.
\bibitem{SZ} T. \v{S}al\'{a}t and \v{S} Zn\'{a}m, On sums of the prime powers, {\it Acta Fac. Rerum Natur. Univ. Comenian. Math.} 21 (1968), 21–24 (1969).
\end{thebibliography}
\end{document}